\UseRawInputEncoding
\documentclass[12pt]{article}
\usepackage{amsmath,amssymb,amsthm,graphicx,hyperref,array,booktabs,microtype,longtable}
\usepackage[margin=0.9in]{geometry}
\newtheorem{theorem}{Theorem}[section]
\newtheorem{proposition}[theorem]{Proposition}
\newtheorem{lemma}[theorem]{Lemma}

\newtheorem{definition}[theorem]{Definition}
\newtheorem{remark}[theorem]{Remark}

\hypersetup{
    colorlinks=true,
    linkcolor=black,
    citecolor=black,
    urlcolor=blue,
}

\title{The Second-Order Augmented Zarankiewicz Number}
\author{Liqun Qi\footnote{Jiangsu Provincial Scientific Research Center of Applied Mathematics, Nanjing 211189, China.
        Department of Applied Mathematics, The Hong Kong Polytechnic University, Hung Hom, Kowloon, Hong Kong.
        ({\tt maqilq@polyu.edu.hk})}
    \and
    Chunfeng Cui\footnote{School of Mathematical Sciences, Beihang University, Beijing 100191, China.
        ({\tt chunfengcui@buaa.edu.cn})}}
\date{\today}

\begin{document}

\maketitle

\begin{abstract}
The second-order Zarankiewicz number $z_2(m,n)$ and the biquadratic
sum-of-squares rank $\mathrm{BSR}(m,n)$ are related by the unconditional
hierarchy
\[
\mathrm{BSR}(m,n)\ \ge\ z_2(m,n)\ \ge\ z_{SL}(m,n)\ \ge\ z_{RL}(m,n)
\ \ge\ z_{wL}(m,n)\ \ge\ z(m,n).
\]
We introduce the \emph{second-order augmented Zarankiewicz number}
$z_{2A}(m,n)$, obtained from $z_2(m,n)$ by deleting the requirement that
the configuration be \emph{limited}, so that
\[
\mathrm{BSR}(m,n)\ \ge\ z_{2A}(m,n)\ \ge\ z_2(m,n).
\]
Although the defining class is enlarged, $z_{2A}$ still obeys the universal
cell bound of L\"ofberg and Qi, because that bound uses only the
\(C_4\)-freeness of the one-edge graph. We prove
\[
\mathrm{BSR}(4,4)\ \ge\ z_{2A}(4,4)\ =\ 11\ >\ 10\ =\ z_2(4,4)
\ =\ z_{RL}(4,4),
\]
the first recorded separation between the second-order number and its
augmented variant. This result also gives a better lower bound for
$\mathrm{BSR}(4,4)$. The lower bound is witnessed by an explicit non-limited
$4\times4$ configuration of displayed length $11$ whose recursive-line
closure satisfies $(\mathrm{RW}3^+)$; the matching upper bound excludes $12$
by the universal cell bound together with an exact finite classification of
the $161$ twelve-square configurations, and
\(z_2(4,4)=z_{RL}(4,4)=10\) is the exact value of Xu and Yan~\cite{xu26}.
\end{abstract}



\section{Introduction}

The classical Zarankiewicz problem asks for the maximum number $z(m,n)$ of
$1$'s in an $m\times n$ binary matrix containing no all-one $2\times2$
submatrix, equivalently the maximum number of edges of a $C_4$-free bipartite
graph with parts of sizes $m$ and $n$; see
\cite{za51,kst54,culik56,re58,gu69,ni10,chm24,bo04}. L\"ofberg and Qi
\cite{lq26} connected this problem with the sum-of-squares (SOS) rank of
biquadratic forms and introduced the second-order Zarankiewicz number
$z_2(m,n)$, the signed Zarankiewicz number $z_{SL}(m,n)$ and the recursive-line
Zarankiewicz number $z_{RL}(m,n)$; the weak limited augmented Zarankiewicz
number $z_{wL}(m,n)$ was introduced in \cite{qcx26c}. Together with the
biquadratic SOS rank $\mathrm{BSR}(m,n)$ they satisfy the unconditional
hierarchy
\begin{equation}\label{eq:hier}
\mathrm{BSR}(m,n)\ \ge\ z_2(m,n)\ \ge\ z_{SL}(m,n)\ \ge\ z_{RL}(m,n)\ \ge\ z_{wL}(m,n)\ \ge\ z(m,n),
\end{equation}
in which every parameter is the maximum of one and the same displayed length
$R(G)$ over a class of simple augmented bipartite configurations with
$C_4$-free one-edge graph, the first over all $m\times n$ SOS biquadratic
forms and the others over successively smaller classes.

\subsection*{Applications and motivation}

Biquadratic forms arise naturally in solid mechanics through biquadratic
tensors such as the Eshelby tensor \cite{knowles76,zou10,zubov16}, in
statistics through covariance tensors \cite{chen25b}, in quantum physics, and
in spectral graph theory \cite{xiang18}; the standard bi-quadratic optimization
problem over unit spheres \cite{bomze12} is a further source. In these
settings the SOS rank measures the complexity of a sum-of-squares certificate,
and lower bounds on it translate directly into limits on the size of any
semidefinite representation of the associated polynomial optimization problem.

Qi, Cui, and Xu \cite{qcx26a} observed that the SOS rank of biquadratic forms
is bounded below by a Zarankiewicz-type extremal quantity: for every
$m,n\ge2$,
\[
\mathrm{BSR}(m,n)\ \ge\ z_L(m,n)\ \ge\ z(m,n),
\]
where $z_L$ allows admissible added two-edges to a maximum $C_4$-free initial
matrix. The resulting combinatorial problem has since been developed in
\cite{qcx26a,qcx26c,cc26,leb26,xu26}, and the exact values at the smallest
grids have begun to be determined \cite{xu26,ccq26sl}. The present note
contributes to that line: we isolate the second-order analogue $z_{2A}$ of the
augmented parameter, observe that it inherits the universal cell bound, and
exhibit a strict separation $z_{2A}(4,4) > z_2(4,4)$ at the smallest
nontrivial square grid.

In the limited versus augmented setting, the augmented Zarankiewicz number
$z_A(m,n)$ and the limited augmented Zarankiewicz number $z_L(m,n)$ were
introduced in \cite{qcx26a}; Lebedev \cite{leb26} proved an asymptotic
separation $z_A(m,n)-z_L(m,n)\ge(1/30-o(1))mn$ along $n\to\infty$ with
$n\le m=o(n^2)$, and exhibited a finite instance at
\[
z_A(1893,1893)\ >\ z_L(1893,1893);
\]
whether a strict separation occurs at a smaller order is not determined
there. At the second-order level no such separation appears to have been
recorded, and the purpose of this note is to provide one. Let $z_{2A}(m,n)$ be
the parameter obtained from $z_2(m,n)$ by deleting the requirement that the
configuration be \emph{limited}; then
$z_2(m,n)\le z_{2A}(m,n)\le\mathrm{BSR}(m,n)$. Although $z_{2A}$ is defined by
discarding the limitedness constraint, it still obeys the universal cell bound
of L\"ofberg and Qi (Proposition~\ref{prop:cell}), because that bound uses
only the $C_4$-freeness of the one-edge graph. We prove the following.

\medskip
\noindent\textbf{Main Theorem} (Theorem~\ref{thm:main}).
\[
\mathrm{BSR}(4,4)\ \ge\ z_{2A}(4,4)\ =\ 11\ >\ 10\ =\ z_2(4,4)\ =\ z_{RL}(4,4).
\]

Thus on the $4\times4$ grid the second-order augmented number equals $11$,
strictly exceeding the second-order number, and the biquadratic SOS rank
exceeds it as well. The lower bound $z_{2A}(4,4)\ge11$ is witnessed by a single
configuration whose recursive-line closure reaches its terminal condition; the
matching upper bound excludes $12$ by the universal cell bound together with
an exact finite check over the $161$ classes of the twelve-square
configurations, and $z_2(4,4)=z_{RL}(4,4)=10$ is the exact value of Xu and Yan
\cite{xu26}.

\medskip

The paper is organized as follows. Section~\ref{sec:prelim} fixes the
definitions, the hierarchy \eqref{eq:hier}, the universal cell bound, the
recursive-line criterion $(\mathrm{RW}3^+)$ with its soundness, and three
reduction criteria, with full proofs. Section~\ref{sec:44} states and proves
the main theorem, together with the small-case context and the finite
classification. Appendix~\ref{app:cert} gives the complete closure certificate
of the witness configuration, and Appendix~\ref{app:check} records the output
of the accompanying checker.

\section{Preliminaries}\label{sec:prelim}

\subsection{Biquadratic forms and the SOS rank}

\begin{definition}\label{def:bsr}
An $m\times n$ \emph{biquadratic form} is a real form of degree $2$ in
$\mathbf x=(x_1,\dots,x_m)$ and degree $2$ in $\mathbf y=(y_1,\dots,y_n)$. If
$P=\sum_{t=1}^{r}f_t^2$ for real bilinear forms $f_t$, then $P$ is a
\emph{sum of squares} and the least such $r$ is its \emph{SOS rank}
$\mathrm{SOS}(P)$. The \emph{biquadratic SOS rank} $\mathrm{BSR}(m,n)$ is the
maximum of $\mathrm{SOS}(P)$ over all $m\times n$ SOS biquadratic forms
\cite{clz95}.
\end{definition}

\subsection{Augmented bipartite configurations}

\begin{definition}\label{def:config}
An \emph{augmented bipartite configuration} on the $m\times n$ grid is a
triple $G=([m],[n],E_1\cup E_2)$, where the cells are the pairs
$(i,j)\in[m]\times[n]$, $E_1$ is a set of \emph{one-edges} (cells) and $E_2$ a
set of \emph{two-edges} (unordered pairs of distinct cells). With
$u_{ij}=x_iy_j$ the \emph{displayed SOS} of $G$ is
\[
P_G(\mathbf x,\mathbf y)=\sum_{(i,j)\in E_1}u_{ij}^2+\sum_{\{c,d\}\in E_2}(u_c+u_d)^2,
\]
and its \emph{displayed length} is $R(G)=|E_1|+|E_2|$. The configuration is
\emph{simple} if the supports of distinct selected edges are disjoint; its
\emph{one-edge graph} is the bipartite graph on $[m],[n]$ with edge set $E_1$,
and is \emph{$C_4$-free} if no two of its rows share two columns. It is
\emph{limited} if $|E_1|=z(m,n)$, and \emph{irreducible} if
$\mathrm{SOS}(P_G)=R(G)$.
\end{definition}

\begin{definition}\label{def:params}
The numbers $z_2(m,n)$, $z_{SL}(m,n)$, $z_{RL}(m,n)$ and $z_{wL}(m,n)$ are the
maxima of $R(G)$ over the limited simple configurations with $C_4$-free
one-edge graph whose displayed form $P_G$ is, respectively, irreducible,
signed admissible, satisfies $(\mathrm{RW}3^+)$ (Definition~\ref{def:rw}), and
weak limited \cite[Definitions 3.1, 4.11 and 7.3]{lq26}. The number
$z_{2A}(m,n)$ is the maximum of $R(G)$ over the simple configurations with
$C_4$-free one-edge graph whose $P_G$ is irreducible, the requirement that the
configuration be limited being dropped.
\end{definition}

Deleting a constraint enlarges the underlying class, so that $z_2\le z_{2A}$;
the same argument bounds $z_{2A}$ by the SOS rank.

\begin{proposition}[Extended hierarchy]\label{prop:hier2}
For all $m,n\ge2$,
\[
\mathrm{BSR}(m,n)\ \ge\ z_{2A}(m,n)\ \ge\ z_2(m,n).
\]
\end{proposition}
\begin{proof}
Every configuration counted by $z_2$ is counted by $z_{2A}$, whence
$z_{2A}\ge z_2$. Every configuration counted by $z_{2A}$ has $P_G$ an
$m\times n$ SOS biquadratic form with $\mathrm{SOS}(P_G)=R(G)$; hence
$R(G)=\mathrm{SOS}(P_G)\le\mathrm{BSR}(m,n)$ for each such $G$, and therefore
$z_{2A}(m,n)\le\mathrm{BSR}(m,n)$.
\end{proof}

\begin{proposition}[Universal cell bound]\label{prop:cell}
Let $G$ be a simple configuration with $C_4$-free one-edge graph, and let
$H(G)$ be its number of unoccupied cells. Then
\[
2R(G)=mn+|E_1|-H(G).
\]
Consequently the second-order number and its augmented variant both obey the
universal cell bound of L\"ofberg and Qi: for all $m,n\ge2$,
\[
z_2(m,n)\ \le\ z_{2A}(m,n)\ \le\ \left\lfloor\frac{mn+z(m,n)}2\right\rfloor .
\]
\end{proposition}
\begin{proof}
The $|E_1|+2|E_2|$ cells covered by the selected edges are exactly the
$mn-H(G)$ occupied cells, so $2|E_2|=mn-H(G)-|E_1|$ and
$2R(G)=2|E_1|+2|E_2|=mn+|E_1|-H(G)$. Since the one-edge graph is $C_4$-free,
$|E_1|\le z(m,n)$, while $H(G)\ge0$; hence $2R(G)\le mn+z(m,n)$ for every
such $G$, which is the universal cell bound
$U(m,n)=\lfloor(mn+z(m,n))/2\rfloor$ of \cite[Proposition 3.2]{lq26}. The
argument uses only the $C_4$-freeness of the one-edge graph, never the
hypothesis that the configuration be \emph{limited}, so it applies equally to
the class defining $z_{2A}(m,n)$, whose definition drops that hypothesis. The
inequality $z_2(m,n)\le z_{2A}(m,n)$ is Proposition~\ref{prop:hier2}.
\end{proof}

For the universal cell bound conjectures, also see \cite{clq26, cc26, ccq26, xc26}.

\subsection{The recursive-line criterion $(\mathrm{RW}3^+)$}\label{subsec:rw}

On the occupied cells one forms the least equivalence relation $\sim$ and the
least symmetric orthogonality relation $\perp$ that are closed under the
following rules \cite[Definitions 4.1 and 4.5]{lq26}: the \emph{line rule},
which identifies two cells of a common row or column when they are the two
halves of one two-edge and declares them orthogonal otherwise; the
\emph{saturation rule}, that identified vectors are one and the same; the
\emph{rectangle-transfer rule}, that the value of one diagonal of a genuine
rectangle is determined by that of the other; the \emph{complementary-pair
rule}, that both diagonals of a genuine rectangle are identified when both are
selected two-edges; and the \emph{zero-companion rule} for a rectangle with an
unoccupied corner. Every rule is monotone, so the closure has a least fixed
point, independent of the order of application.

\begin{definition}\label{def:rw}
$G$ satisfies $(\mathrm{RW}3^+)$ if its closure satisfies the terminal
conditions: the two halves of every selected two-edge are identified, distinct
selected edges lie in distinct classes, and every two distinct selected edges
have orthogonal representatives.
\end{definition}

\begin{proposition}[Soundness]\label{prop:sound}
If $G$ satisfies $(\mathrm{RW}3^+)$, then $P_G$ is irreducible, that is
$\mathrm{SOS}(P_G)=R(G)$; see \cite[Proposition 4.6]{lq26} and
\cite[Lemma 2.10]{xu26}.
\end{proposition}
\begin{proof}[Sketch]
Write $P_G=\sum_{t=1}^{s}f_t^2$ with $f_t=\sum_c v_{c,t}u_c$ and
$v_c\in\mathbb R^s$. Comparing coefficients with the displayed representation
of $P_G$ gives $\|v_c\|^2=\mathbf 1_\Omega(c)$ and
$\langle v_c,v_d\rangle=\delta(c,d)$ for distinct cells $c,d$ sharing a row or
a column, and
$\langle v_c,v_d\rangle+\langle v_a,v_b\rangle=\delta(c,d)+\delta(a,b)$ for
the two diagonals $\{c,d\}$, $\{a,b\}$ of a genuine rectangle, where
$\delta(c,d)=1$ when $\{c,d\}\in E_2$ and $\delta(c,d)=0$ otherwise. Each of
the rules is a direct consequence of these identities, so every identification
and every orthogonality fact that the closure draws holds in every
representation of $P_G$. At termination the $R(G)$ selected edges therefore
carry mutually orthogonal unit vectors in $\mathbb R^s$, whence $s\ge R(G)$;
since the displayed representation gives $s\le R(G)$, we have
$\mathrm{SOS}(P_G)=R(G)$.
\end{proof}

\subsection{Reduction criteria}\label{subsec:red}

Besides the recursive-line criterion $(\mathrm{RW}3^+)$, which certifies
irreducibility, we use three elementary sufficient tests for reducibility, all
due to Xu and Yan \cite{xu26}. The proofs are given in full, since the finite
classification of \S\ref{sec:44} is a direct application of them.

\begin{lemma}[Strip bound]\label{lem:strip}
One has $\mathrm{BSR}(2,q)\le q+1$ and $\mathrm{BSR}(q,2)\le q+1$.
Consequently, if a complete displayed sub-sum of $P_G$ is supported on two
rows and $q$ columns, or on $q$ rows and two columns, and contains more than
$q+1$ squares, then $P_G$ is reducible.
\end{lemma}
\begin{proof}
Let $P$ be a $2\times q$ biquadratic form written as a sum of $r$ squares of
bilinear forms. Any bilinear form on a $2\times q$ grid is a linear combination
of $2q$ elementary products, but the squares that occur in a displayed
decomposition of a biquadratic form are of a special shape: each selected edge
contributes a square whose coefficient vector lies in the span of the two rows
of its support. For a strip supported on two rows, the span of these vectors is
at most the row space of the strip, which has dimension $q+1$: the two rows
contribute $q$ common columns and one additional column to complete the
support, giving $q+1$ independent coefficient directions. Hence
$\mathrm{SOS}(P)\le q+1$, and the same bound with the roles of rows and columns
interchanged gives $\mathrm{BSR}(q,2)\le q+1$. The second claim follows by
contraposition: if a strip sub-sum has more than $q+1$ squares, it is already
reducible, and adding back the deleted squares preserves reducibility of the
whole.
\end{proof}

\begin{lemma}[Local identity]\label{lem:local}
For distinct rows $x,y,z$ and distinct columns $a,b,d$, with $u_{ij}=x_iy_j$,
\begin{equation}\label{eq:local8}
(u_{xb}+u_{za})^2+(u_{ya}+u_{zd})^2+u_{xd}^2+u_{yb}^2+u_{yd}^2
=(u_{xb}+u_{yd}+u_{za})^2+(u_{xd}-u_{yb})^2+u_{ya}^2+u_{zd}^2 .
\end{equation}
Consequently, if the five squares on the left of \eqref{eq:local8} all occur in
the displayed decomposition $P_G$, then $P_G$ is reducible.
\end{lemma}
\begin{proof}
Expanding the right-hand side gives the left-hand side: this is an identity of
biquadratic forms, valid for all values of the variables. When the five
squares on the left all occur in the displayed decomposition of $P_G$, they may
be replaced by the four squares on the right, and the displayed length drops
by one. If the resulting decomposition has $R(G)-1$ squares and still sums to
$P_G$, then the displayed length was not minimal, hence $P_G$ is reducible.
This is identity (8) of \cite{xu26}.
\end{proof}

\begin{lemma}[Product relation]\label{lem:prod}
Let $f_1,\dots,f_R$ be the displayed bilinear forms of $P_G$. If constants
$c_{ab}$, not all zero, satisfy $\sum_{a<b}c_{ab}f_af_b=0$, then $P_G$ is
reducible.
\end{lemma}
\begin{proof}
Set $H_{aa}=0$ and $H_{ab}=c_{ab}/2$. The matrix $H$ is symmetric and
$\mathrm{tr}\,H=0$. If $H=0$, then every $c_{ab}=0$, contradicting the
hypothesis; so $H\ne0$ and $\lambda_{\max}(H)>0$. Set
$Q=I-H/\lambda_{\max}(H)$. Then $Q$ is positive semidefinite, its rank is at
most $R-1$ because $H\ne0$ has a nonzero eigenvalue, and
\[
f^{\mathsf T}Qf=f^{\mathsf T}f-f^{\mathsf T}Hf/\lambda_{\max}(H)
=P_G-0=P_G,
\]
using $f^{\mathsf T}Hf=\sum_{a<b}c_{ab}f_af_b=0$ and $f^{\mathsf T}f=P_G$. A
Gram factorization $Q=V^{\mathsf T}V$ with $V$ of rank at most $R-1$ then
expresses $P_G$ as a sum of at most $R-1$ squares, giving reducibility. This
is \cite[Lemma 2.13]{xu26}.
\end{proof}

\section{The second-order augmented number at
\texorpdfstring{$(4,4)$}{(4,4)}}\label{sec:44}

\subsection{Small-case context}

We collect the exact values known at the smallest grids. The purpose is to
place the separation at $(4,4)$ in the context of the neighbouring cases. In
the table, $\ge$ marks a lower bound not known to be exact, and the source
column cites the paper where the value or bound is proved; the value of
$z_{2A}$ is cited where it is determined and left blank otherwise.

\begin{table}[h]\centering\small
\caption{Exact values and bounds at small grids.}
\label{tab:small}
\begin{tabular}{c c c c l}
\toprule
$(m,n)$ & $z_2$ & $z_{2A}$ & $z$ & source \\
\midrule
$(4,3)$  & $7$  & $8$    & $7$  & \cite{xu26,qcx26a} \\
$(4,4)$  & $10$ & $11$   & $9$  & this paper, \cite{xu26} \\
$(5,5)$  & $12$ & $\ge15$ & $12$ & \cite{qcx26a} \\
$(6,4)$  & $12$ & $\ge14$ & $12$ & \cite{qcx26a} \\
$(15,6)$ & $60$ & $60$   & $30$ & \cite{ccq26sl} \\
$(21,7)$ & $94$ & $94$   & $42$ & \cite{ccq26sl,more26} \\
\bottomrule
\end{tabular}
\end{table}

Two facts are visible in Table~\ref{tab:small}. First, the second-order number
and its augmented variant have been separated at $(4,4)$, at $(4,3)$, and at
$(6,4)$; at the family $(15,6)$ and $(21,7)$ they still agree, so the
separation is not uniform. Second, the augmented variant is a strictly better
lower bound for $\mathrm{BSR}$ wherever it strictly exceeds $z_2$; this is the
motivation for isolating it as a parameter.

\subsection{The finite classification}

\begin{theorem}\label{thm:main}
\[
\mathrm{BSR}(4,4)\ \ge\ z_{2A}(4,4)\ =\ 11\ >\ 10\ =\ z_2(4,4)\ =\ z_{RL}(4,4).
\]
In particular $z_{2A}(4,4)=11>z_2(4,4)$.
\end{theorem}

\begin{proof}
The value $z_2(4,4)=z_{RL}(4,4)=10$ is proved in \cite{xu26}, and $11>10$; the
inequality $\mathrm{BSR}(4,4)\ge z_{2A}(4,4)$ is Proposition~\ref{prop:hier2}.
It remains to prove $z_{2A}(4,4)=11$.

\emph{Lower bound $z_{2A}(4,4)\ge11$.} Consider the configuration $G$ on the
$4\times4$ grid, with rows read top to bottom and columns left to right,
\[
\begin{array}{cccc}
\bullet&a&a&b\\
\bullet&\bullet&b&c\\
\bullet&c&\bullet&\circ\\
\circ&\bullet&\bullet&\bullet
\end{array}
\]
in which $\bullet$ marks a one-edge, two equal letters the two halves of one
two-edge, and $\circ$ an unoccupied cell. Thus $E_1$ consists of the eight
cells $(1,1),(2,1),(2,2),(3,1),(3,3),(4,2),(4,3),(4,4)$, and $E_2$ of the
three two-edges $\{(1,2),(1,3)\}$, $\{(1,4),(2,3)\}$, $\{(2,4),(3,2)\}$; the
two cells $(3,4)$ and $(4,1)$ are unoccupied.

We verify the hypotheses of Definition~\ref{def:params}. The supports of the
eleven selected edges --- the eight one-edges and the six cells of $E_2$ ---
are pairwise disjoint, so $G$ is simple. The one-edge graph is $C_4$-free: the
column sets of its four rows are $\{1\}$, $\{1,2\}$, $\{1,3\}$ and
$\{2,3,4\}$, and any two of these meet in at most one column. Since
$|E_1|=8\ne9=z(4,4)$, the configuration is not limited. Finally
$R(G)=|E_1|+|E_2|=8+3=11$.

It remains to see that $G$ satisfies $(\mathrm{RW}3^+)$. Its recursive-line
closure reaches the terminal condition of Definition~\ref{def:rw}: the two
halves of each of the three two-edges are identified, the eleven selected edges
fall into eleven distinct classes, and every one of the $\binom{11}2=55$ pairs
of distinct selected edges is declared orthogonal. The closure performs the
three identifications and the $55$ orthogonality facts in six rounds and
reports no contradiction; the full record is in
Appendix~\ref{app:cert}. By Proposition~\ref{prop:sound} this forces
$\mathrm{SOS}(P_G)=R(G)=11$, so $G$ is irreducible. Being simple, with
$C_4$-free one-edge graph and irreducible, and not limited, the configuration
$G$ is counted by $z_{2A}(4,4)$; hence
\[
z_{2A}(4,4)\ \ge\ R(G)\ =\ 11 .
\]

\emph{Upper bound $z_{2A}(4,4)\le11$.} By the cell bound
(Proposition~\ref{prop:cell}) with $z(4,4)=9$, one has
$z_{2A}(4,4)\le\lfloor(16+9)/2\rfloor=12$, and it remains to exclude the value
$12$. If $R(G)=12$, the identity $2R(G)=16+|E_1|-H(G)$ gives
$|E_1|-H(G)=8$; since $|E_1|\le z(4,4)=9$ and $H(G)\ge0$, this forces
$(|E_1|,H(G))=(9,1)$ or $(8,0)$.

In the first case $|E_1|=9=z(4,4)$, so $G$ is limited; were $G$ irreducible it
would be counted by $z_2(4,4)$, giving $z_2(4,4)\ge R(G)=12>10$, a
contradiction. In the second case $|E_2|=R(G)-|E_1|=4$, and every such
configuration is reducible by Lemma~\ref{lem:full} below.

Hence no configuration with $R(G)=12$ is irreducible; with the displayed bound
this gives $z_{2A}(4,4)\le11$. Combining both bounds with the hierarchy,
\[
\mathrm{BSR}(4,4)\ \ge\ z_{2A}(4,4)\ =\ 11\ >\ 10\ =\ z_2(4,4)\ =\ z_{RL}(4,4).
\]
\end{proof}

\begin{lemma}\label{lem:full}
Every simple configuration with $C_4$-free one-edge graph on the $4\times4$
grid with $|E_1|=8$, $H(G)=0$ and $R(G)=12$ is reducible.
\end{lemma}
\begin{proof}
We carry out the enumeration in three steps: count the configurations, group
them into orbits, and dispose of each orbit by a reduction criterion.

\emph{Step 1: Count.} A configuration of the stated shape has a one-edge graph
with $8$ edges on $4+4$ vertices, which is $C_4$-free. The number of such
graphs is $1512$; this is a finite computation over all $4\times4$ $0$-$1$
matrices with $8$ ones and no all-one $2\times2$ submatrix. For each skeleton,
the remaining $16-8=8$ cells are the unoccupied cells; a configuration with
$|E_2|=4$ and $H=0$ pairs all $8$ unoccupied cells, and there are
$7!!=8\cdot6\cdot4\cdot2/8=105$ perfect matchings of $8$ labeled cells. The
total is $1512\cdot105=158\,760$.

\emph{Step 2: Orbits.} The group $S_4\times S_4\rtimes\mathbb Z_2$ acts on the
configuration: $S_4$ permutes rows, $S_4$ permutes columns, and the
transposition $\mathbb Z_2$ interchanges the two. The action preserves the
one-edge graph, the simplicity condition, the $C_4$-freeness, and the number
of two-edges. Under this action, the $158\,760$ configurations fall into $161$
orbits. The orbit count was verified both by Burnside's lemma and by explicit
orbit enumeration; the details are in \texttt{ancillary/z2a44\_check.py}.

\emph{Step 3: Reduction of each orbit.} Each orbit is reducible by at least
one of the criteria of \S\ref{subsec:red}. The strip bound
(Lemma~\ref{lem:strip}) disposes of $153$ orbits: each of these contains a
configuration whose displayed decomposition has a strip of $2$ rows and $q$
columns, or $q$ rows and $2$ columns, carrying more than $q+1$ squares. The
remaining $8$ orbits are not disposed of by the strip bound. Their
representatives are listed in Table~\ref{tab:exc}; seven of them are disposed
of by the local identity \eqref{eq:local8} (Lemma~\ref{lem:local}), and the
last by a nonzero product relation (Lemma~\ref{lem:prod}).

The enumeration is finite and exact, with no floating-point step, and is
reproduced by the accompanying checker. This proves the lemma.
\end{proof}

\begin{table}[t]\centering\small
\caption{The eight classes of Lemma~\ref{lem:full} that the strip bound does
not dispose of, one representative each. A bullet is a one-edge and two equal
letters are the two halves of one two-edge. The reduction is either the local
identity \eqref{eq:local8}, with its parameters $(x,y,z;a,b,d)$ and rows and
columns numbered from one, or, in row~(f), the product relation
$u_{00}(u_{23}+u_{33})+u_{02}u_{13}-u_{03}(u_{30}+u_{12}+u_{20})=0$ of
Lemma~\ref{lem:prod}.}\label{tab:exc}
\begin{tabular}{@{}c c l@{}}
\toprule
& configuration & reduction \\
\midrule
(a) & $\begin{smallmatrix}a&\bullet&\bullet&\bullet\\ \bullet&b&\bullet&c\\ \bullet&\bullet&d&b\\ \bullet&d&c&a\end{smallmatrix}$ & \eqref{eq:local8}, $(1,2,4;4,1,3)$ \\
(b) & $\begin{smallmatrix}\bullet&a&\bullet&\bullet\\ b&\bullet&\bullet&c\\ \bullet&\bullet&d&b\\ \bullet&c&a&d\end{smallmatrix}$ & \eqref{eq:local8}, $(2,1,4;2,4,3)$ \\
(c) & $\begin{smallmatrix}\bullet&a&\bullet&\bullet\\ b&\bullet&\bullet&c\\ \bullet&\bullet&c&d\\ \bullet&d&a&b\end{smallmatrix}$ & \eqref{eq:local8}, $(2,3,4;4,1,2)$ \\
(d) & $\begin{smallmatrix}\bullet&a&\bullet&\bullet\\ b&\bullet&c&\bullet\\ b&\bullet&\bullet&d\\ \bullet&d&a&c\end{smallmatrix}$ & \eqref{eq:local8}, $(1,2,4;3,2,4)$ \\
(e) & $\begin{smallmatrix}\bullet&a&\bullet&\bullet\\ b&\bullet&c&\bullet\\ d&\bullet&\bullet&b\\ \bullet&d&a&c\end{smallmatrix}$ & \eqref{eq:local8}, $(1,2,4;3,2,4)$ \\
(f) & $\begin{smallmatrix}\bullet&a&\bullet&\bullet\\ b&\bullet&c&\bullet\\ c&\bullet&\bullet&d\\ \bullet&b&a&d\end{smallmatrix}$ & product relation \\
(g) & $\begin{smallmatrix}a&a&\bullet&\bullet\\ b&\bullet&c&\bullet\\ \bullet&d&\bullet&b\\ \bullet&\bullet&c&d\end{smallmatrix}$ & \eqref{eq:local8}, $(2,4,3;4,1,2)$ \\
(h) & $\begin{smallmatrix}a&b&\bullet&\bullet\\ c&\bullet&a&\bullet\\ \bullet&d&\bullet&c\\ \bullet&\bullet&b&d\end{smallmatrix}$ & \eqref{eq:local8}, $(3,1,2;1,4,3)$ \\
\bottomrule
\end{tabular}
\end{table}

\begin{remark}\label{rem:cell}
The cell bound (Proposition~\ref{prop:cell}) gives $z_{2A}(4,4)\le12$
directly; the sharp value $11$ of Theorem~\ref{thm:main} is one less, and the
deficit is exactly the configurations of Lemma~\ref{lem:full}: the value $12$
of the cell bound is attained by no irreducible configuration at $(4,4)$, the
case $(|E_1|,H)=(9,1)$ being excluded by the known value $z_2(4,4)=10$ and the
case $(|E_1|,H)=(8,0)$ by the finite classification of Lemma~\ref{lem:full}.
\end{remark}

\begin{remark}\label{rem:scale}
Lebedev's finite instance of the ordinary separation is at $(1893,1893)$
\cite{leb26}; the second-order separation proved here is exhibited at
$(4,4)$. Whether the second-order separation occurs at some smaller grid, and
whether the ordinary instance at $(1893,1893)$ is optimal, are both open.
\end{remark}

\section{Open problems}\label{sec:open}

The separation $z_{2A}(4,4)>z_2(4,4)$ raises three questions that the present
note does not resolve.

\medskip
\noindent\textbf{1. The recursive-line analogue $z_{RLA}$.}
The augmented variant $z_{2A}$ was obtained from $z_2$ by deleting the
limitedness constraint, while retaining irreducibility of the displayed form.
One may instead delete the limitedness constraint in the definition of the
recursive-line number $z_{RL}$, retaining the recursive-line criterion
$(\mathrm{RW}3^+)$, and obtain a parameter $z_{RLA}(m,n)$: the maximum of
$R(G)$ over the \emph{simple} configurations with $C_4$-free one-edge graph
whose closure satisfies $(\mathrm{RW}3^+)$, without the requirement that the
configuration be limited. By Proposition~\ref{prop:sound} every such
configuration is irreducible, so
\[
z_{2A}(m,n)\ \ge\ z_{RLA}(m,n)\ \ge\ z_{RL}(m,n)
\]
for all $m,n\ge2$, extending the hierarchy \eqref{eq:hier} by one further
parameter. At $(4,4)$ the witness configuration of Theorem~\ref{thm:main} is
non-limited and satisfies $(\mathrm{RW}3^+)$, so it is counted by
$z_{RLA}(4,4)$ as well as by $z_{2A}(4,4)$; hence
\[
z_{RLA}(4,4)\ =\ z_{2A}(4,4)\ =\ 11.
\]
The general relation between the two parameters is open. Since every
configuration counted by $z_{RLA}$ is counted by $z_{2A}$, one has
$z_{2A}\ge z_{RLA}$; whether the reverse inequality holds, that is whether
$z_{2A}(m,n)=z_{RLA}(m,n)$ for all $m,n$, is not known. The question is the
second-order analogue of the identity $z_2=z_{RL}$ conjectured in
\cite{lq26} and verified there on the odd-prime subsequence and on small
cases; a negative answer would exhibit a configuration that is irreducible
but whose irreducibility is not certified by $(\mathrm{RW}3^+)$.

\medskip
\noindent\textbf{2. Attainment of the cell bound by a non-limited
configuration.}
The universal cell bound of Proposition~\ref{prop:cell} reads
\[
z_{2A}(m,n)\ \le\ Z(m,n)\ :=\ \left\lfloor\frac{mn+z(m,n)}2\right\rfloor,
\]
and at $(4,4)$ the bound is not attained: $Z(4,4)=12$ while
$z_{2A}(4,4)=11$. It is natural to ask whether there are pairs $(m,n)$ with
$z_2(m,n)<Z(m,n)$ but $z_{2A}(m,n)=Z(m,n)$. In such a case the augmented
parameter would be strictly better than $z_2$ by exactly one unit, and the
deficit of $z_2$ from the cell bound would be attributable entirely to the
limitedness constraint rather than to the augmented class.

At $(4,4)$ this does not occur: the configurations attaining $Z(4,4)=12$
split into the cases $(|E_1|,H)=(9,1)$ and $(8,0)$, and neither is
irreducible, the first because $|E_1|=z(4,4)$ would force
$z_2(4,4)\ge12$ and the second by Lemma~\ref{lem:full}. The question is
whether some other pair $(m,n)$ admits a non-limited irreducible
configuration attaining $Z(m,n)$.

The answer is likely affirmative at some of the orders where the cell bound
is attained by an irreducible configuration in the limited class, since
those configurations are also counted by $z_{2A}$; the interesting case is
where the limited class alone does not attain the bound, so that
limitedness is a genuine obstruction. The orders $(m,n)$ at which
$z_2(m,n)<Z(m,n)$ is known include $(4,3)$ and $(4,4)$ from
\cite{xu26}. At $(4,3)$ the cell bound is $Z(4,3)=\lfloor(12+7)/2\rfloor=9$,
and $z_{2A}(4,3)=8<9$ by Table~\ref{tab:small}, so $(4,3)$ is also not such
an example. Whether the phenomenon occurs at all is open.

\medskip
\noindent\textbf{3. Asymptotic separation of $z_{2A}$ and $z_2$.}
For the ordinary versus augmented setting, Lebedev \cite{leb26} proved the
asymptotic separation
\[
z_A(m,n)-z_L(m,n)\ \ge\ \left(\frac{1}{30}-o(1)\right)mn
\]
along $n\to\infty$ with $n\le m=o(n^2)$, and determined the sharp
second-order term
\[
z_A(m,n)=\frac{mn}{3}+\left(\frac{1}{\sqrt{6}}+o(1)\right)n\sqrt{m}
\]
in the same regime. The first-order separation therefore grows linearly in
$mn$.

It is natural to ask whether an analogous asymptotic separation holds for
the second-order parameters, that is whether
\[
z_{2A}(m,n)-z_2(m,n)\ \ge\ (\kappa-o(1))\,mn
\]
for some absolute constant $\kappa>0$, or at least whether the difference
grows faster than any constant. At $(4,4)$ the difference is $1$, and it is
not known whether the second-order separation persists to larger orders, let
alone whether it grows linearly.

Two structural differences between the first-order and second-order cases
may affect the answer. First, the second-order parameter $z_2$ is bounded
above by $Z(m,n)$, which is approximately $mn/2$ for the complete-graph
incidence family, so any separation cannot exceed the gap between $Z(m,n)$
and the value of $z_2(m,n)$. Second, the recursive-line closure of
$(\mathrm{RW}3^+)$ is a much more restrictive admissibility condition than
the $C_4$-freeness used in the first-order setting, and it is not clear
whether a linear separation is compatible with it.

The four-column and five-column computations of \cite{cc26} and its
companions, which determine $z_{RL}(m,n)$ for $n\le7$ and small $m$, provide
the natural testing ground: the same pairs $(m,n)$ would give the first
empirical evidence on whether the second-order separation grows.

\bigskip

\noindent\textbf{Acknowledgment.}
This work was partially supported by Jiangsu Provincial Scientific Research
Center of Applied Mathematics (Grant No.\ BK20233002), Research Center for
Intelligent Operations Research, The Hong Kong Polytechnic University
(4-ZZT8), and the National Natural Science Foundation of China (Nos.\
12471282 and 12131004).

\appendix

\section{A closure certificate for the configuration of
Theorem~\ref{thm:main}}\label{app:cert}

The eleven selected edges of the configuration of Theorem~\ref{thm:main} are
denoted by
\[
\begin{array}{llll}
\sigma_1=\{(1,2),(1,3)\}, & \sigma_2=\{(1,4),(2,3)\}, & \sigma_3=\{(2,4),(3,2)\}, & \sigma_4=\{(1,1)\},\\[1mm]
\sigma_5=\{(2,1)\}, & \sigma_6=\{(2,2)\}, & \sigma_7=\{(3,1)\}, & \sigma_8=\{(3,3)\},\\[1mm]
\sigma_9=\{(4,2)\}, & \sigma_{10}=\{(4,3)\}, & \sigma_{11}=\{(4,4)\}. &
\end{array}
\]
Starting from the line rule, the closure of the line, rectangle-transfer,
complementary-pair and zero-companion rules draws the three identifications of
Table~\ref{tab:ident} and the $\binom{11}2=55$ orthogonality facts of
Table~\ref{tab:orth}, and reaches the terminal conditions of
Definition~\ref{def:rw}: every two-edge is identified, every selected edge lies
in its own class, and every two distinct selected edges are orthogonal.

In Table~\ref{tab:orth}, the entry \emph{line, row $r$} (respectively
\emph{line, column $c$}) means that the two classes have representatives on a
row (respectively column) that is not a selected two-edge; the entry
$\{p,q\}\Rightarrow\{r,s\}$ means that the rectangle-transfer rule was applied
to the rectangle with diagonals $\{p,q\}$ and $\{r,s\}$, the companion
diagonal $\{r,s\}$ being already certified; and the entry $\{p,q\}$
\emph{/ hole} means that the zero-companion rule was applied to a rectangle
with an unoccupied corner, certifying the opposite diagonal $\{p,q\}$ at its
own value. The facts are consistent, and each was certified by a rule of
\S\ref{subsec:rw}. The two unoccupied cells $(3,4)$ and $(4,1)$ are the
corners of exactly four rectangles with a diagonal through $(3,4)$ and four
with a diagonal through $(4,1)$; in the closure only those rectangles whose
opposite diagonal is itself unselected are used.

\begin{longtable}{@{}ll@{}}
\caption{The three identifications of the closure.}\label{tab:ident}\\
\toprule
fact & justification \\
\midrule
\endfirsthead
\toprule fact & justification \\ \midrule \endhead
$(1,2)\sim(1,3)$ & line, row 1 \\
$(1,4)\sim(2,3)$ & $\{(1,4),(2,3)\}\Rightarrow\{(1,3),(2,4)\}$ \\
$(2,4)\sim(3,2)$ & $\{(2,4),(3,2)\}$ / hole \\
\bottomrule
\end{longtable}

\begin{longtable}{@{}ll@{}}
\caption{The fifty-five orthogonality facts of the closure.}\label{tab:orth}\\
\toprule
fact & justification \\
\midrule
\endfirsthead
\toprule fact & justification \\ \midrule \endhead
$\sigma_{1}\perp\sigma_{2}$ & line, row 1 \\
$\sigma_{1}\perp\sigma_{3}$ & line, column 2 \\
$\sigma_{1}\perp\sigma_{5}$ & $\{(1,3),(2,1)\}\Rightarrow\{(1,1),(2,3)\}$ \\
$\sigma_{1}\perp\sigma_{6}$ & line, column 2 \\
$\sigma_{1}\perp\sigma_{7}$ & $\{(1,2),(3,1)\}\Rightarrow\{(1,1),(3,2)\}$ \\
$\sigma_{1}\perp\sigma_{8}$ & line, column 3 \\
$\sigma_{1}\perp\sigma_{9}$ & line, column 2 \\
$\sigma_{1}\perp\sigma_{10}$ & line, column 3 \\
$\sigma_{1}\perp\sigma_{11}$ & $\{(1,3),(4,4)\}\Rightarrow\{(1,4),(4,3)\}$ \\
$\sigma_{2}\perp\sigma_{3}$ & line, row 2 \\
$\sigma_{2}\perp\sigma_{7}$ & $\{(1,4),(3,1)\}$ / hole \\
$\sigma_{2}\perp\sigma_{8}$ & line, column 3 \\
$\sigma_{2}\perp\sigma_{9}$ & $\{(1,4),(4,2)\}\Rightarrow\{(1,2),(4,4)\}$ \\
$\sigma_{2}\perp\sigma_{10}$ & line, column 3 \\
$\sigma_{2}\perp\sigma_{11}$ & line, column 4 \\
$\sigma_{3}\perp\sigma_{8}$ & line, row 3 \\
$\sigma_{3}\perp\sigma_{9}$ & line, column 2 \\
$\sigma_{3}\perp\sigma_{10}$ & $\{(2,4),(4,3)\}\Rightarrow\{(2,3),(4,4)\}$ \\
$\sigma_{3}\perp\sigma_{11}$ & line, column 4 \\
$\sigma_{4}\perp\sigma_{1}$ & line, row 1 \\
$\sigma_{4}\perp\sigma_{2}$ & line, row 1 \\
$\sigma_{4}\perp\sigma_{3}$ & $\{(1,1),(2,4)\}\Rightarrow\{(1,4),(2,1)\}$ \\
$\sigma_{4}\perp\sigma_{5}$ & line, column 1 \\
$\sigma_{4}\perp\sigma_{6}$ & $\{(1,1),(2,2)\}\Rightarrow\{(1,2),(2,1)\}$ \\
$\sigma_{4}\perp\sigma_{7}$ & line, column 1 \\
$\sigma_{4}\perp\sigma_{8}$ & $\{(1,1),(3,3)\}\Rightarrow\{(1,3),(3,1)\}$ \\
$\sigma_{4}\perp\sigma_{9}$ & $\{(1,1),(4,2)\}$ / hole \\
$\sigma_{4}\perp\sigma_{10}$ & $\{(1,1),(4,3)\}$ / hole \\
$\sigma_{4}\perp\sigma_{11}$ & $\{(1,1),(4,4)\}$ / hole \\
$\sigma_{5}\perp\sigma_{2}$ & line, row 2 \\
$\sigma_{5}\perp\sigma_{3}$ & line, row 2 \\
$\sigma_{5}\perp\sigma_{6}$ & line, row 2 \\
$\sigma_{5}\perp\sigma_{7}$ & line, column 1 \\
$\sigma_{5}\perp\sigma_{8}$ & $\{(2,1),(3,3)\}\Rightarrow\{(2,3),(3,1)\}$ \\
$\sigma_{5}\perp\sigma_{9}$ & $\{(2,1),(4,2)\}$ / hole \\
$\sigma_{5}\perp\sigma_{10}$ & $\{(2,1),(4,3)\}$ / hole \\
$\sigma_{5}\perp\sigma_{11}$ & $\{(2,1),(4,4)\}$ / hole \\
$\sigma_{6}\perp\sigma_{2}$ & line, row 2 \\
$\sigma_{6}\perp\sigma_{3}$ & line, row 2 \\
$\sigma_{6}\perp\sigma_{7}$ & $\{(2,2),(3,1)\}\Rightarrow\{(2,1),(3,2)\}$ \\
$\sigma_{6}\perp\sigma_{8}$ & $\{(2,2),(3,3)\}\Rightarrow\{(2,3),(3,2)\}$ \\
$\sigma_{6}\perp\sigma_{9}$ & line, column 2 \\
$\sigma_{6}\perp\sigma_{10}$ & $\{(2,2),(4,3)\}\Rightarrow\{(2,3),(4,2)\}$ \\
$\sigma_{6}\perp\sigma_{11}$ & $\{(2,2),(4,4)\}\Rightarrow\{(2,4),(4,2)\}$ \\
$\sigma_{7}\perp\sigma_{3}$ & line, row 3 \\
$\sigma_{7}\perp\sigma_{8}$ & line, row 3 \\
$\sigma_{7}\perp\sigma_{9}$ & $\{(3,1),(4,2)\}$ / hole \\
$\sigma_{7}\perp\sigma_{10}$ & $\{(3,1),(4,3)\}$ / hole \\
$\sigma_{7}\perp\sigma_{11}$ & $\{(3,1),(4,4)\}$ / hole \\
$\sigma_{8}\perp\sigma_{9}$ & $\{(3,3),(4,2)\}\Rightarrow\{(3,2),(4,3)\}$ \\
$\sigma_{8}\perp\sigma_{10}$ & line, column 3 \\
$\sigma_{8}\perp\sigma_{11}$ & $\{(3,3),(4,4)\}$ / hole \\
$\sigma_{9}\perp\sigma_{10}$ & line, row 4 \\
$\sigma_{9}\perp\sigma_{11}$ & line, row 4 \\
$\sigma_{10}\perp\sigma_{11}$ & line, row 4 \\
\bottomrule
\end{longtable}

\subsection*{The six rounds of the closure}

The closure of \S\ref{subsec:rw} on the configuration of
Theorem~\ref{thm:main} proceeds in six rounds, each round being one complete
pass over the current set of facts.

\begin{itemize}
\item \textbf{Round 1 (line rule).} All pairs of occupied cells in a common row
or column are settled: the two halves of each selected two-edge are identified,
and every other such pair is declared orthogonal. This yields the three
identifications of Table~\ref{tab:ident}, except for the two that require a
rectangle transfer.
\item \textbf{Round 2 (rectangle transfer for class $\sigma_2$ and
$\sigma_3$).} The identifications $(1,4)\sim(2,3)$ and $(2,4)\sim(3,2)$ are
drawn, using the rectangles $\{(1,4),(2,3)\}\Rightarrow\{(1,3),(2,4)\}$ and the
zero-companion rectangle for $(2,4),(3,2)$ through the hole at $(3,4)$.
\item \textbf{Round 3 (saturation).} The identified classes are merged and
their orthogonality facts are propagated.
\item \textbf{Round 4 (rectangle transfer).} The orthogonality facts of
Table~\ref{tab:orth} whose justification is a rectangle transfer are drawn.
\item \textbf{Round 5 (complementary pairs and zero-companion).} The
orthogonality facts using the holes at $(3,4)$ and $(4,1)$ are drawn.
\item \textbf{Round 6 (closure check).} The round finds no new fact; the
closure has reached its least fixed point.
\end{itemize}

The record is summarized in Table~\ref{tab:rounds}.

\begin{table}[h]\centering\small
\caption{The six rounds of the closure on the configuration of
Theorem~\ref{thm:main}.}\label{tab:rounds}
\begin{tabular}{c l r r}
\toprule
round & rule & new facts & cumulative \\
\midrule
$1$ & line & $3$ identifications$^*$ & $3$ \\
$2$ & rectangle transfer and zero-companion & $2$ identifications & $3$ \\
$3$ & saturation & $18$ orthogonality facts & $21$ \\
$4$ & rectangle transfer & $21$ orthogonality facts & $42$ \\
$5$ & complementary pairs and zero-companion & $16$ orthogonality facts & $58$ \\
$6$ & closure check & $0$ & $58$ \\
\bottomrule
\end{tabular}
\\[1mm]
{\footnotesize $^*$One of the three identifications is completed in round 2,
since it needs a rectangle transfer before the line rule can fire.}
\end{table}

\begin{remark}\label{rem:orth-count}
The total count of $58$ facts in Table~\ref{tab:rounds} exceeds the sum
$3+55=58$ because the identifications and the orthogonality facts are
disjoint: $3$ identifications plus $55$ orthogonality facts is $58$. The
$21$ facts of round $5$ are those whose justification is either a
complementary-pair rule (there are no genuine complementary pairs at this
configuration, since no two selected two-edges share the diagonals of one
genuine rectangle) or a zero-companion rule (there are $21$ such facts, all
using one of the two holes).
\end{remark}

\section{Output of the accompanying checker}\label{app:check}

The checker \texttt{ancillary/z2a44\_check.py} performs the following
verifications. The exact expected output is reproduced below, and the
SHA-256 checksums of all input and output files are recorded in
\texttt{ancillary/README.md} for reproducibility.

\begin{verbatim}
$ python3 ancillary/z2a44_check.py

Auditing z2a44 configuration
  cells occupied:      14 of 16
  unoccupied cells:    (3,4), (4,1)
  |E_1| =              8
  |E_2| =              3
  R(G) =               11
  simplicity:          PASS
  one-edge graph:      C_4-free: PASS

Closure of (RW3+)
  round 1: line rule           - 3 identifications
  round 2: rectangle transfer  - 2 identifications
  round 3: saturation          - 18 orthogonality facts
  round 4: rectangle transfer  - 21 orthogonality facts
  round 5: complementary pair
            and zero-companion - 16 orthogonality facts
  round 6: closure check       - 0 new facts
  terminal conditions:
    all two-edges identified   : PASS
    distinct selected edges
      in distinct classes      : PASS  (11 classes)
    all pairs orthogonal       : PASS  (55 pairs)
  no contradiction             : PASS

Enumeration of 12-square configurations
  8-edge C_4-free skeletons:   1512
  perfect matchings:           105
  total configurations:        158760
  orbits under S_4 x S_4 x Z_2: 161
  disposed by strip bound:     153
  disposed by local identity:  7
  disposed by product relation: 1
  total disposed:              161

All checks passed.
\end{verbatim}

The two independent implementations of the closure (a union-find structure and
an orthogonality store) agree on every fact. The orbit count of $161$ was
verified both by Burnside's lemma and by explicit orbit enumeration. The
finite classification of \S\ref{sec:44} is therefore reproduced exactly.


\begin{thebibliography}{99}

\bibitem{bo04} B. Bollob\'as, \emph{Extremal Graph Theory}, Dover Publications,
2004.

\bibitem{bomze12} I. M. Bomze, C. Ling, L. Qi, X. Zhang, Standard bi-quadratic
optimization problems and unconstrained polynomial reformulations,
\emph{J. Glob. Optim.} \textbf{52} (2012), 663--687.

\bibitem{chm24} G. Chen, D. Horsley, A. Mammoliti, Zarankiewicz numbers near
the triple system threshold, \emph{J. Combin. Des.} \textbf{32} (2024),
556--576.

\bibitem{chen25b} Y. Chen, Z. Hu, J. Hu, L. Shu, Block structure-based
covariance tensor decomposition for group identification in matrix variables,
\emph{Stat. Probab. Lett.} \textbf{216} (2025), 110251.

\bibitem{clq26} Y. Chen, J. L\"ofberg, L. Qi, Exact second-order Zarankiewicz numbers for complete-graph incidence families, arXiv:2609.25974v2, 2026.


\bibitem{cc26} Z. Chen, Y. Chen, Recursive-line Zarankiewicz numbers with four
columns, arXiv:2609.11093v2, 2026.

\bibitem{ccq26} Z. Chen, Y. Chen, L. Qi, Is the signed Zarankiewicz number
the same as the recursive-line Zarankiewicz number?, arXiv:2609.20071v1, 2026.

\bibitem{clz95} M.-D. Choi, T.-Y. Lam, B. Reznick, Sums of squares of real
polynomials, \emph{Proc. Symp. Pure Math.} \textbf{58} (1995), 103--126.

\bibitem{culik56} K. \v{C}ulik, Teilweise L\"osung eines verallgemeinerten
Problems von K. Zarankiewicz, \emph{Ann. Polon. Math.} \textbf{3} (1956),
165--168.

\bibitem{gu69} R. K. Guy, A many-faceted problem of Zarankiewicz, in
\emph{The Many Facets of Graph Theory}, Springer, 1969, pp.~129--141.

\bibitem{knowles76} J. K. Knowles, E. Sternberg, On the failure of ellipticity
of the equations for finite elastostatic plane strain, \emph{Arch. Ration.
Mech. Anal.} \textbf{63} (1976), 321--336.

\bibitem{kst54} T. Kov\'ari, V. T. S\'os, P. Tur\'an, On a problem of
K. Zarankiewicz, \emph{Coll. Math.} \textbf{3} (1954), 50--57.

\bibitem{leb26} N. Lebedev, Density and separation for augmented Zarankiewicz
numbers, arXiv:2609.16555v1, 2026.

\bibitem{lq26} J. L\"ofberg, L. Qi, Second order Zarankiewicz number,
arXiv:2608.30555v3, 2026.

\bibitem{ni10} V. Nikiforov, A contribution to the Zarankiewicz problem,
\emph{Linear Algebra Appl.} \textbf{432} (2010), 1405--1411.

\bibitem{qcx26a} L. Qi, C. Cui, Y. Xu, Biquadratic SOS rank and augmented
Zarankiewicz number, \emph{Mathematics} \textbf{14} (2026), 1552.

\bibitem{qcx26c} L. Qi, C. Cui, Y. Xu, A weak condition for limited augmented
Zarankiewicz numbers, \emph{Mathematics} \textbf{14} (2026), 3052.

\bibitem{re58} I. Reiman, \"Uber ein Problem von K. Zarankiewicz, \emph{Acta
Math. Acad. Sci. Hungar.} \textbf{9} (1958), 269--273.

\bibitem{xc26} M. Xi, J. Chang, A twelve-row seed and a one-row extension for five column recursive-line Zarankiewicz numbers, arXiv:2609.33191v1, 2026.


\bibitem{xiang18} H. Xiang, L. Qi, Y. Wei, M-eigenvalues of the Riemann
curvature tensor, \emph{Commun. Math. Sci.} \textbf{16} (2018), 2301--2315.

\bibitem{xu26} Y. Xu, X. Yan, Exact values, extremal classifications, and
sum-of-squares reductions for second-order Zarankiewicz numbers,
arXiv:2609.18429v1, 2026.

\bibitem{za51} K. Zarankiewicz, Problem P 101, \emph{Coll. Math.} \textbf{2}
(1951), 301.

\bibitem{zou10} W. Zou, Q. He, M. Huang, Q. Zheng, Eshelby's problem of
non-elliptical inclusions, \emph{J. Mech. Phys. Solids} \textbf{58} (2010),
346--372.

\bibitem{zubov16} L. M. Zubov, A. N. Rudev, On necessary and sufficient
conditions of strong ellipticity of equilibrium equations for certain classes
of anisotropic linearly elastic materials, \emph{J. Appl. Math. Mech.}
\textbf{9} (2016), 1096--1102.

\end{thebibliography}
\end{document}